\documentclass{birkjour}
\usepackage{hyperref,color,amssymb}

\newcommand{\cA}{\mathcal{A}}
\newcommand{\cB}{\mathcal{B}}
\newcommand{\cK}{\mathcal{K}}
\newcommand{\cM}{\mathcal{M}}
\newcommand{\fM}{\mathfrak{M}}

\newcommand{\eps}{\varepsilon}
\newcommand{\dR}{{\bf\dot{\R}}}

\newcommand{\C}{\mathbb{C}}
\newcommand{\N}{\mathbb{N}}
\newcommand{\R}{\mathbb{R}}

\newtheorem{theorem}{Theorem}[section]
\newtheorem{lemma}[theorem]{Lemma}
\newtheorem{corollary}[theorem]{Corollary}
\newtheorem{question}[theorem]{Question}
\theoremstyle{definition}

\theoremstyle{remark}
\newtheorem{remark}[theorem]{Remark}
\numberwithin{equation}{section}

\begin{document}

%
%
%
%
%
%
%
%
%

\title[Algebras of convolution type operators]
{Algebras of convolution type\\ 
operators with continuous data\\ 
do not always contain\\ 
all rank one operators}
\author[A. Karlovich]{Alexei Karlovich}
\address{
Centro de Matem\'atica e Aplica\c{c}\~oes\\
Departamento de Matem\'atica\\
Faculdade de Ci\^encias e Tecnologia\\
Universidade Nova de Lisboa\\
Quinta da Torre\\
2829--516 Caparica\\
Portugal} \email{oyk@fct.unl.pt}
\author[E. Shargorodsky]{Eugene Shargorodsky}
\address{%
Department of Mathematics\\
King's College London\\
Strand, London WC2R 2LS\\
United Kingdom
\\ and \\
Technische Universit\"at Dresden\\
Fakult\"at Mathematik\\
01062 Dresden\\
Germany}
\email{eugene.shargorodsky@kcl.ac.uk}



\date{January 1, 2004}
\dedicatory{{
Dedicated to Bernd Silbermann on the occasion of his 80th birthday
}}
\begin{abstract}
Let $X(\mathbb{R})$ be a separable Banach function space such that the 
Hardy-Littlewood maximal operator is bounded {on} $X(\mathbb{R})$ 
and on its associate space $X'(\mathbb{R})$. The algebra $C_X(\dR)$
of continuous Fourier multipliers on $X(\mathbb{R})$ is defined as the 
closure of the set of continuous functions of bounded variation on 
$\dR=\mathbb{R}\cup\{\infty\}$ with respect to the multiplier norm. It 
was proved by C. Fernandes, Yu. Karlovich and the first author
\cite{FKK19} that if the space $X(\R)$ is reflexive, then the ideal of compact 
operators is contained in the Banach algebra $\mathcal{A}_{X(\mathbb{R})}$ 
generated by all multiplication operators $aI$ by continuous functions 
$a\in C(\dR)$ and by all Fourier convolution operators $W^0(b)$ with symbols 
$b\in C_X(\dR)$. We show that there are separable and non-reflexive 
Banach function spaces $X(\mathbb{R})$ such that the algebra 
$\mathcal{A}_{X(\mathbb{R})}$ does not contain all rank one operators. 
In particular, this happens in the case of the Lorentz spaces 
$L^{p,1}(\mathbb{R})$ with $1<p<\infty$.
\end{abstract}

\maketitle
\section{Introduction}
We denote by $\mathcal{S}(\R)$ the Schwartz class of all infinitely 
differentiable and rapidly decaying functions (see, e.g., 
\cite[Section~2.2.1]{G14}). Let $F$ denote the Fourier transform, defined
on $\mathcal{S}(\R)$ by
\[
(Ff)(x):=\widehat{f}(x):=\int_\R f(t)e^{itx}\,dt,
\quad
x\in\R,
\]
and let $F^{-1}$ be the inverse of $F$ defined on $\mathcal{S}(\R)$ 
by
\[
(F^{-1}g)(t)=\frac{1}{2\pi}\int_\R g(x)e^{-itx}\,d x,
\quad
t\in\R.
\]
It is well known that these operators extend uniquely to the space
$L^2(\R)$. As usual, we will use 
{the} symbols $F$ and $F^{-1}$ for the
direct and inverse Fourier transform on $L^2(\R)$. 
The Fourier convolution operator 
\[
W^0(a):=F^{-1}aF
\]
is bounded on the space $L^2(\R)$ for every $a\in L^\infty(\R)$.

In this 
{paper, we study} algebras of operators generated by 
operators of multiplication $aI$ and Fourier convolution operators
$W^0(b)$ on so-called Banach function spaces in the case when both $a$ and $b$
are continuous. We postpone a formal definition of a Banach function 
space $X(\R)$ and its associate space $X'(\R)$ until Section~\ref{sec:BFS}. 
The Lebesgue spaces $L^p(\R)$ with 
$1\le p\le\infty$ constitute the most important example of Banach function
spaces. The class of Banach function spaces includes classical Orlicz
spaces $L^\Phi(\R)$, Lorentz spaces $L^{p,q}(\R)$, all other 
rearrangement-invariant spaces, as well 
{as} (non-rearrangement-in\-vari\-ant) 
weighted Lebesgue spaces $L^p(\R,w)$ and 
variable Lebesgue spaces $L^{p(\cdot)}(\R)$.

Let $X(\R)$ be a separable Banach function space. Then $L^2(\R)\cap X(\R)$
is dense in $X(\R)$ (see Lemma~\ref{le:density} below). A function 
$a\in L^\infty(\R)$ is called a Fourier multiplier on $X(\R)$ if the 
convolution operator $W^0(a):=F^{-1}aF$ maps $L^2(\R)\cap X(\R)$ into 
$X(\R)$ and extends to a bounded linear operator on $X(\R)$. The function 
$a$ is called the symbol of the Fourier convolution operator $W^0(a)$. 
The set $\cM_{X(\R)}$ of all Fourier multipliers on  $X(\R)$ is a unital 
normed algebra under pointwise operations and the norm
\[
\left\|a\right\|_{\cM_{X(\R)}}:=\left\|W^0(a)\right\|_{\cB(X(\R))},
\]
where $\cB(X(\R))$ denotes the Banach algebra of all bounded linear operators
on the space $X(\R)$. Let $\cK(X(\R))$ denote the ideal of all compact
operators in the Banach algebra $\cB(X(\R))$.

Recall that the (non-centered) Hardy-Littlewood maximal function $Mf$ of a
function $f\in L_{\rm loc}^1(\R)$ is defined by
\[
(Mf)(x):=\sup_{Q\ni x}\frac{1}{|Q|}\int_Q|f(y)|\,dy,
\]
where the supremum is taken over all intervals $Q\subset\R$ of finite length
containing $x$. The Hardy-Littlewood maximal operator $M$ defined by the rule 
$f\mapsto Mf$ is a sublinear operator.

Suppose that $a:\R\to\C$ is a function of bounded variation $V(a)$ given 
by
\[
V(a):=\sup \sum_{k=1}^n |a(x_k)-a(x_{k-1})|,
\]
where the supremum is taken over all partitions of $\R$ of the form
\[
-\infty<x_0<x_1<\dots<x_n<+\infty
\]
with $n\in\N$. The set $V(\R)$ of all functions of bounded variation
on $\R$ with the norm
\[
\|a\|_{V(\R)}:=\|a\|_{L^\infty(\R)}+V(a)
\]
is a unital non-separable Banach algebra.

Let $X(\R)$ be a separable Banach function space such that the
Hardy-Littlewood maximal operator $M$ is bounded on $X(\R)$ and on its
associate space $X'(\R)$. It follows from \cite[Theorem~4.3]{K15a}
that if a function $a:\R\to\C$ has a bounded variation $V(a)$, then 
the convolution operator $W^0(a)$ is bounded on the space $X(\R)$ and
\begin{equation}\label{eq:Stechkin}
\|W^0(a)\|_{\cB(X(\R))}
\le
c_{X}\|a\|_{{V(\R)},}
\end{equation}
where $c_{X}$ is a positive constant depending only on $X(\R)$.

For Lebesgue spaces $L^p(\R)$, $1<p<\infty$, inequality~\eqref{eq:Stechkin} is
usually called Stechkin's inequality.
We refer to \cite[Theorem~2.11]{D79} for the proof of \eqref{eq:Stechkin}
in the case of Lebesgue spaces $L^p(\R)$ with $c_{L^p}=\|S\|_{\cB(L^p(\R))}$,
where $S$ is the Cauchy singular integral operator.

Let $C(\dR)$ denote the $C^*$-algebra of continuous functions on the one-point 
compactification $\dR=\R\cup\{\infty\}$ of the real line. For a subset 
$\mathfrak{S}$ of a Banach space $\mathcal{E}$, we denote by
$\operatorname{clos}_\mathcal{E}(\mathfrak{S})$ the closure of $\mathfrak{S}$
with respect to the norm of $\mathcal{E}$. Consider the following algebra of 
continuous Fourier multipliers:
\begin{equation}\label{eq:algebra-CX}
C_{X}(\dR):=\operatorname{clos}_{\cM_{X(\R)}}
\big(C(\dR)\cap V(\R)\big).
\end{equation}
It follows Theorem \ref{th:continuous-embedding} below that 
$C_X(\dR)\subset C(\dR)$. 
The aim of this paper is to continue the study of the smallest Banach subalgebra
\[
\cA_{X(\R)}
:=
\operatorname{alg}\{aI,W^0(b)\ :\ a\in C(\dR),\ b\in C_X(\dR)\}
\]
of the algebra $\cB(X(\R))$ that contains all operators of multiplication $aI$ 
by functions $a\in C(\dR)$ and all Fourier convolution operators $W^0(b)$ with 
symbols $b\in C_X(\dR)$ started in the setting of reflexive Banach function
spaces in \cite{FKK19}. The main result of that paper says the following.
\begin{theorem}[{\cite[Theorem~1.1]{FKK19}}]
\label{th:FKK}
Let $X(\R)$ be a reflexive Banach function space such that the Hardy-Littlewood 
maximal operator $M$ is bounded on $X(\R)$ and on its associate space 
$X'(\R)$. Then the ideal of compact operators $\cK(X(\R))$ is contained in the 
Banach algebra $\cA_{X(\R)}$.
\end{theorem}
Note that results of this kind are well known in the setting of (weighted)
Lebesgue spaces (see, e.g., \cite[Lemma~6.1]{KILH13a}, \cite[Theorem~5.2.1 and 
Proposition~5.8.1]{RSS11} and also \cite[Lemma~8.23]{BK97}, 
\cite[Theorem~4.1.5]{RSS11}). They constitute the first step in the Fredholm 
study of more general algebras of convolution type operators with more 
general function algebras in place of $C(\dR)$ and $C_X(\dR)$, respectively 
(see, e.g., \cite{KILH12,KILH13a,KILH13b}), by means of local principles 
(see, e.g., \cite[Sections~1.30--1.35]{BS06}).

Let $\mathfrak{A}$ be a Banach algebra with unit $e$.  The center 
$\operatorname{Cen}\mathfrak{A}$ of $\mathfrak{A}$ is the set of all elements 
$z\in\mathfrak{A}$ with the property that $za = az$ for all $a\in\mathfrak{A}$.
One can successfully apply the Allan-Douglas local principle 
\cite[Section~1.35]{BS06} to the algebra $\mathfrak{A}$ if it possesses 
a (hopefully large) closed subalgebra $\mathfrak{C}$ lying in its center.
Having applications of the Allan-Douglas local principle in mind,
the authors of \cite{FKK19} asked whether the quotient algebra
\[
\cA_{X(\R)}^\pi:= \cA_{X(\R)}/\cK(X(\R))
\]
is commutative under the assumptions of Theorem~\ref{th:FKK}. Our first
result is the positive answer to \cite[Question~1.2]{FKK19}.
\begin{theorem}
\label{th:algebra-Api}
Let $X(\R)$ be a reflexive Banach function space such that the Hardy-Littlewood 
maximal operator $M$ is bounded on $X(\R)$ and on its associate space 
$X'(\R)$. Then the quotient algebra $\cA_{X(\R)}^\pi$ is commutative.
\end{theorem}
It is well known that a Banach function space $X(\R)$ is reflexive if and 
only if the space $X(\R)$ and its associate space $X'(\R)$ are separable 
(see 
{
\cite[Chap.~1, \S2, Theorem 4 and \S3, Corollary 1 to Theorem 7]{L55} or
} 
\cite[Chap.~1, Corollaries 4.4 and 5.6]{BS88}). So, it is natural to 
ask whether the assumption of the reflexivity of the space $X(\R)$ in 
Theorem~\ref{th:FKK} can be relaxed to the assumption of the separability 
of the space $X(\R)$. Our main result says that this is impossible.
\begin{theorem}[Main result]
\label{th:main}
There exists a separable non-reflexive Banach function space $X(\R)$ such that
\begin{enumerate}
\item[(a)]
the Hardy-Littlewood maximal operator is bounded on $X(\R)$ and on 
its associate space $X'(\R)$;
\item[(b)]
the algebra $\cA_{X(\R)}$ does not contain all rank one operators.
\end{enumerate}
\end{theorem}
This theorem means that 
{the} usual methods of the Fredholm study of algebras
of convolution type operators with discontinuous data on non-reflexive
separable Banach function spaces will require a modification to overcome
an obstacle that certain compact operators do not belong to the algebra
$\cA_{X(\R)}$ and, therefore, the quotient algebra $\cA_{X(\R)}/\cK(X(\R))$
cannot be defined. 

In fact, 
{
Theorem~\ref{th:main} holds for a familiar example of separable and 
non-reflexive Banach function spaces, namely the classical Lorentz 
spaces $L^{p,1}(\R)$ with $1<p<\infty$. }
Let us recall their 
definition. The distribution function $\mu_f$ of a
measurable function $f:\R\to\C$ is given by
\[
\mu_f(\lambda):=|\{x\in\R:|f(x)|>\lambda\}|,
\quad\lambda\ge 0.
\]
The non-increasing
rearrangement of $f$ is the function $f^*$ defined on $[0,\infty)$ by
\[
f^*(t)=\inf\{\lambda:\mu_f(\lambda)\le t\},
\quad t\ge 0
\]
(see, e.g., \cite[Chap.~3, Definitions 1.1 and 1.5]{BS88}).

For given $1 < p < \infty$ and $1 \le q \le \infty$, the Lorentz
space $L^{p, q}(\mathbb{R})$ consist of all measurable 
functions  $f : \mathbb{R} \to \mathbb{C}$ such that the norm
\[
\|f\|_{(p, q)} 
:=  
\left\{\begin{array}{cl}
\displaystyle\left(\int_0^\infty \left(t^{1/p} f^{**}(t)\right)^q\, 
\frac{dt}{t}\right)^{1/q},     
&  q < \infty,
\\[3mm]
\displaystyle\sup_{0 < t < \infty} \left(t^{1/p} f^{**}(t)\right),   
&   q = \infty,
\end{array}\right.
\]
is finite, where 
\[
{
f^{**}(t) := \frac1t\, \int_0^t f^*(x)\, dx
}
\]
(see \cite[Chap.~4, Lemma 4.5]{BS88}).
\begin{theorem}\label{th:main-Lorentz}
Let $1<p<\infty$. The Lorentz space $L^{p,1}(\R)$ is a separable
and non-reflexive Banach function space satisfying assumption {\rm(a)}
of Theorem~{\rm\ref{th:main}} and such that the algebra $\cA_{L^{p,1}(\R)}$ 
does not contain all rank one operators.
\end{theorem}
The paper is organized as follows. 
In Section~\ref{sec:auxiliary}, we collect definitions of a Banach function
space and its associate space $X'(\R)$, recall that the set of Fourier
multipliers $\cM_{X(\R)}$ on a separable Banach function space $X(\R)$, 
such that the Hardy-Littlewood maximal operator $M$ is bounded on $X(\R)$
and on its associate space $X'(\R)$, is continuously embedded into 
$L^\infty(\R)$. Consequently, $\cM_{X(\R)}$ is a unital Banach algebra.
Further, we prove several lemmas on approximation of continuous
functions (or Fourier multipliers) vanishing at infinity by
compactly supported continuous functions (or Fourier multipliers, respectively).

In Section~\ref{sec:commutativity}, we show that if $X(\R)$ is a
separable Banach function space  such that the Hardy-Littlewood maximal 
operator $M$ is bounded on $X(\R)$ and on its associate space $X'(\R)$
and $a\in C(\dR)$, $b\in C_X(\dR)$, then the commutator
$aW^0(b)-W^0(b)aI$ is compact on the space $X(\R)$. Combining this result 
with Theorem~\ref{th:FKK}, we arrive at Theorem~\ref{th:algebra-Api}.

Section~\ref{sec:main-proof} is devoted to the proof of 
Theorems~\ref{th:main} and~\ref{th:main-Lorentz}. 
For $R>0$, let $\chi_{\{R\}}:=\chi_{\R\setminus[-R,R]}$.
We show that
if $a$ is a compactly supported continuous function and
$b$ is a compactly supported function of bounded variation,
then the norm of the operator $aW^0(b)\chi_{\{R\}}I$  goes to zero as 
$R\to\infty$. If a Banach function space $X(\R)$ is separable and 
non-reflexive, its associate space $X'(\R)$ may contain a function
$g$ such that $\|g\chi_{\{R\}}\|_{X'(\R)}$ is bounded away from
zero for all $R>0$ (this cannot happen if $X(\R)$ is reflexive).
If, in addition, the Hardy-Littlewood operator is bounded on $X(\R)$
and on its associate space $X(\R)$, then we show that 
for every $h\in X(\R)\setminus\{0\}$ the rank one operator
\[
(T_{g,h}f)(x) := h(x) \int_\R g(y) f(y)\, dy 
\]
does not belong to the algebra $\cA_{X(\R)}$, which implies
Theorem~\ref{th:main} under the assumption that the 
function $g\in X'(\R)$  mentioned  above does indeed exist. 
Let $1<p<\infty$ and $1/p+1/p'=1$.
Finally, we prove Theorem~\ref{th:main-Lorentz} first recalling that
the classical Lorentz space $L^{p,1}(\R)$ is a separable non-reflexive 
Banach function space with the associate space $L^{p',\infty}(\R)$, that 
the Hardy-Littlewood maximal operator is bounded on both $L^{p,1}(\R)$ 
and $L^{p',\infty}(\R)$; and then showing that the function 
$g(x)=|x|^{-1/p'}$ belongs to $L^{p',\infty}(\R)$ and 
$\|\chi_{\{R\}}g\|_{(p',\infty)}$ is bounded away from zero for all 
$R>0$. This completes the proof of Theorem~\ref{th:main-Lorentz} and,
thus, of Theorem~\ref{th:main}.

In Section~\ref{sec:SO}, we define the algebra of continuous
Fourier multipliers $C_{X}^0(\dR)$ as the closure of $\C\dot{+} C_c^\infty(\R)$,
where $C_c^\infty(\R)$ is the set of smooth compactly supported functions
and $\C$ denotes the set of constant functions. It is not difficult
to see that $C_X^0(\dR)\subset C_X(\dR)$. We do not know whether
these algebras coincide, in general. We prove a possible refinement of 
Theorem~\ref{th:FKK} for the algebra $\cA^0_{X(\R)}$, where the latter
algebra is defined in the same way as the algebra $\cA_{X(\R)}$
with $C_X(\dR)$ replaced by $C_X^0(\dR)$. Further, we recall
the definition of the set of slowly oscillating functions $SO^\diamond$
and slowly oscillating Fourier multipliers $SO_{X(\R)}^\diamond$
(see \cite{FKK20,K15c}). Since $C(\dR)\subset SO^\diamond$ and
$C_X^0(\dR)\subset SO_{X(\R)}^\diamond$, 
{the} ideal of compact operators $\cK(X(\R))$ is contained
in the algebra $\mathcal{D}_{X(\R)}$ generated by the operators $aI$
with $a\in SO^\diamond$ and $b\in SO_{X(\R)}^\diamond$ under the assumptions
that $X(\R)$ is a reflexive Banach function space such that the 
Hardy-Littlewood maximal operator is bounded on $X(\R)$ and on its associate
space $X'(\R)$. 
{
We conclude the paper with an open question on whether or not 
the quotient algebra 
$\mathcal{D}_{X(\R)}^\pi:=\mathcal{D}_{X(\R)}/\mathcal{K}(X(\R))$ 
is commutative in this case.}
\section{Auxiliary results}\label{sec:auxiliary}
\subsection{Banach function spaces}\label{sec:BFS}
The set of all Lebesgue measurable complex-valued functions on $\R$ is denoted
by $\fM(\R)$. Let $\fM^+(\R)$ be the subset of functions in $\fM(\R)$ whose
values lie  in $[0,\infty]$. For a measurable set $E\subset\R$, 
its Lebesgue measure and the characteristic function are denoted by $|E|$ and
$\chi_E$, respectively. Following 
{\cite[p.~3]{L55}
(see also \cite[Chap.~1, Definition~1.1]{BS88}
and \cite[Definition~6.1.5]{PKJF13}),} a 
mapping $\rho:\fM^+(\R)\to [0,\infty]$ is called a Banach function norm if,
for all functions $f,g, f_n \ (n\in\N)$ in $\fM^+(\R)$, for all
constants $a\ge 0$, and for all measurable subsets $E$ of $\R$,
the following properties hold:
\begin{eqnarray*}
{\rm (A1)} &\quad & \rho(f)=0  \Leftrightarrow  f=0\ \mbox{a.e.}, \
\rho(af)=a\rho(f), \
\rho(f+g) \le \rho(f)+\rho(g),\\
{\rm (A2)} &\quad &0\le g \le f \ \mbox{a.e.} \ \Rightarrow \ \rho(g)
\le \rho(f)
\quad\mbox{(the lattice property)},
\\
{\rm (A3)} &\quad &0\le f_n \uparrow f \ \mbox{a.e.} \ \Rightarrow \
       \rho(f_n) \uparrow \rho(f)\quad\mbox{(the Fatou property)},\\
{\rm (A4)} &\quad & 
{E \text{ is bounded}}  
\Rightarrow \rho(\chi_E) <\infty,\\
{\rm (A5)} &\quad & 
{E \text{ is bounded}}  
\Rightarrow \int_E f(x)\,dx \le C_E\rho(f),
\end{eqnarray*}
where $C_E \in (0,\infty)$ may depend on $E$ and $\rho$ but is
independent of $f$. 
When functions differing only on a set of measure zero
are identified, the set $X(\R)$ of functions $f\in\fM(\R)$
for which $\rho(|f|)<\infty$ is called a Banach function space. For each
$f\in X(\R)$, the norm of $f$ is defined by
\[
\left\|f\right\|_{X(\R)} :=\rho(|f|).
\]
With this norm and under natural linear space operations, the set $X(\R)$ 
becomes a Banach space (see 
{\cite[Chap.~1, \S1, Theorem 1]{L55} or} 
\cite[Chap.~1, Theorems~1.4 and~1.6]{BS88}). 
If $\rho$ is a Banach function norm, its associate norm $\rho'$ is defined on
$\fM^+(\R)$ by
\[
\rho'(g):=\sup\left\{
\int_{\R} f(x)g(x)\,dx \ : \ f\in \fM^+(\R), \ \rho(f) \le 1
\right\}, \quad g\in \fM^+(\R).
\]
{
Then $\rho'$ is itself 
a Banach function norm (see \cite[Chap.~1, \S1]{L55} or 
\cite[Chap.~1, Theorem~2.2]{BS88}).
}
The Banach function space $X'(\R)$ determined by the Banach function norm
$\rho'$ is called the associate space (K\"othe dual) of $X(\R)$.
The associate space $X'(\R)$ is a subspace of the (Banach) dual
space $[X(\R)]^*$.
{
\begin{remark}
We note that our definition of a Banach function space is slightly 
different from that found in \cite[Chap.~1, Definition~1.1]{BS88}
and \cite[Definition~6.1.5]{PKJF13}. 
In particular, in Axioms (A4) and (A5)
we assume that the set $E$ is a bounded set, whereas it is sometimes 
assumed that $E$ merely satisfies $|E| <\infty$. We do this so that the 
weighted Lebesgue spaces with Muckenhoupt weights satisfy
Axioms (A4) and (A5). Moreover, it is well known that
all main elements of the general theory of Banach function spaces
work with (A4) and (A5) stated for bounded sets \cite{L55} (see also 
the discussion at the beginning of Chapter~1 on page~2 of \cite{BS88}).
Unfortunately, we overlooked that
the definition of a Banach function space in our previous works
\cite{FK20,FKK-AFA,FKK19,FKK20,K15a,K15b,K15c,KS19,KS14}
had to be changed by replacing in Axioms (A4) and (A5)
the requirement of $|E|<\infty$ by the requirement 
that $E$ is a bounded set to include weighted Lebesgue spaces
in our considerations. However, the results proved in the above
papers remain true. 
\end{remark}
}
\subsection{Density of nice functions in Banach function spaces}
Let $C_c(\R)$ and $C_c^\infty(\R)$ denote the sets of continuous compactly
supported functions on $\R$ and {of} infinitely differentiable 
compactly supported functions on $\R$, respectively.
\begin{lemma}\label{le:density}
Let $X(\R)$ be a separable Banach function space. Then the sets $C_c(\R)$,
$C_c^\infty(\R)$ and $L^2(\R)\cap X(\R)$ are dense in the space $X(\R)$.
\end{lemma}
The density of $C_c(\R)$ and $C_c^\infty(\R)$ in $X(\R)$ is shown in 
\cite[Lemma~2.12]{KS14}. Since $C_c(\R)\subset L^2(\R)\cap X(\R)\subset X(\R)$,
we conclude that $L^2(\R)\cap X(\R)$ is dense in $X(\R)$.
\subsection{Banach algebra \boldmath{$\cM_{X(\R)}$} of Fourier multipliers}
The following result plays an important role in this paper.
\begin{theorem}
\label{th:continuous-embedding}
Let $X(\R)$ be a separable Banach function space such that the
Hardy-Littlewood maximal operator $M$ is bounded on  $X(\R)$ or
on its associate space $X'(\R)$. If $a\in\cM_{X(\R)}$, then
\begin{equation}\label{eq:continuous-embedding}
\|a\|_{L^\infty(\R)}\le\|a\|_{\cM_{X(\R)}}.
\end{equation}
The constant $1$ on the right-hand side of \eqref{eq:continuous-embedding}
is best possible.
\end{theorem}
\begin{proof}
If the Hardy-Littlewood maximal operator $M$ is bounded on the space $X(\R)$
or on its associate space $X'(\R)$, then in view of \cite[Lemma~3.2]{H12}
we have
\[
\sup_{-\infty<a<b<\infty}
\frac{1}{b-a}\|\chi_{(a,b)}\|_{X(\R)}\|\chi_{(a,b)}\|_{X'(\R)}<\infty.
\]
If this condition is fulfilled, then inequality \eqref{eq:continuous-embedding}
follows from \cite[inequality (1.2) and Corollary~4.2]{KS19}.
\end{proof}
Inequality \eqref{eq:continuous-embedding} was established earlier in
\cite[Theorem~1]{K15b} with some constant on the right-hand side that depends 
on the space $X(\R)$ under the assumption that the operator $M$ is bounded on 
both $X(\R)$ and $X'(\R)$ (see  also \cite[Theorem~2.4]{FKK-AFA}).

Since \eqref{eq:continuous-embedding} is available, an easy adaptation of
the proof of \cite[Proposition 2.5.13]{G14} leads to the following
(we refer to the proof of \cite[Corollary~1]{K15b} for details).
\begin{corollary}
Let $X(\R)$ be a separable Banach function space such that the
Hardy-Littlewood maximal operator $M$ is bounded on $X(\R)$ or
on its associate space $X'(\R)$. Then the set of Fourier multipliers
$\cM_{X(\R)}$ is a Banach algebra under pointwise operations and the norm
$\|\cdot\|_{\cM_{X(\R)}}$.
\end{corollary}
\subsection{Approximation of continuous functions vanishing at infinity}
Let $C_0(\R)$ denote the set of all continuous functions on $\R$ that vanish 
at $\pm\infty$. 
\begin{lemma}\label{le:approximating-vanishing}
For a function $\upsilon \in C_c^\infty(\mathbb{R})$ such that 
$0 \le \upsilon \le 1$ and $\upsilon(x) = 1$ when $|x| \le 1$, let 
\[
\upsilon_n(x) := \upsilon(x/n), \quad x\in\R, \quad n \in \mathbb{N}.
\]
\begin{enumerate}
\item[(a)] If $a\in C_0(\R)$, then
\begin{equation}\label{eq:approximating-vanishing-1}
\lim_{n\to\infty}\|a-\upsilon_n a\|_{L^\infty(\R)}=0.
\end{equation}

\item[(b)] If $a\in C_0(\R)\cap V(\R)$, then
\begin{equation}\label{eq:approximating-vanishing-2}
\lim_{n\to\infty}\|a-\upsilon_n a\|_{V(\R)}=0.
\end{equation}
\end{enumerate}
\end{lemma}
\begin{proof}
(a) If $a\in C_0(\R)$, then for every $\eps>0$ there exists $N\in\N$
such that 
\[
\sup_{x\in\R\setminus[-N,N]}|a(x)|<\frac{\eps}{2}.
\]
For all $n>N$ and $x\in[-N,N]$, we have $v_n(x)=1$. Since $0\le\upsilon_n\le 1$,
for $n>N$, we get
\[
\|a-\upsilon_n a\|_{L^\infty(\R)}=
\sup_{x\in\R\setminus[-N,N]}|a(x)-\upsilon_n(x)a(x)|
\le 2\sup_{x\in\R\setminus[-N,N]}|a(x)|<\eps,
\]
which completes the proof of equality \eqref{eq:approximating-vanishing-1}.

(b) Let $V(g;\Omega)$ denote the total variation of a function $g$ 
over a union of intervals $\Omega \subset \mathbb{R}$. Then for all $n\in\N$,
\begin{align}
V(a - \upsilon_n a) 
=& 
V(a(1 - \upsilon_n);\mathbb{R}\setminus [-n, n]) 
\nonumber\\
\le & 
V(a;\mathbb{R}\setminus [-n, n]) 
\|1 - \upsilon_n\|_{L^\infty(\mathbb{R}\setminus [-n, n])} 
\nonumber\\
&+ 
\|a\|_{L^\infty(\mathbb{R}\setminus [-n, n])} 
V(1 - \upsilon_n;\mathbb{R}\setminus [-n, n]) 
\nonumber\\
\le & 
V(a;\mathbb{R}\setminus [-n, n]) 
+ 
\|a\|_{L^\infty(\mathbb{R}\setminus [-n, n])} V(\upsilon).
\label{eq:approximating-vanishing-3}
\end{align}
Since $a\in C_0(\R)$, we have
\begin{equation}\label{eq:approximating-vanishing-4}
\lim_{n\to\infty}\|a\|_{L^\infty(\R\setminus[-n,n])}=0
\end{equation}
(see the proof of part (a)).
On the other hand,
\begin{align}
\lim_{n\to\infty}V(a;\mathbb{R}\setminus [-n, n])
&=
\lim_{n\to\infty}\big(V(a)-V{(a;[-n,n])}\big)
\nonumber\\
&=V(a)-V(a)=0.
\label{eq:approximating-vanishing-5}
\end{align}
It follows from \eqref{eq:approximating-vanishing-3}--%
\eqref{eq:approximating-vanishing-5} that
\begin{equation}\label{eq:approximating-vanishing-6}
\lim_{n\to\infty}V(a-\upsilon_n a)=0.
\end{equation}
Combining equalities \eqref{eq:approximating-vanishing-1}
and \eqref{eq:approximating-vanishing-6},
we arrive at equality \eqref{eq:approximating-vanishing-2}.
\end{proof}
\subsection{Approximation of continuous Fourier multipliers vanishing 
at {infinity}}
\begin{lemma}\label{le:approximating-vanishing-multipliers}
Let $X(\R)$ be a Banach function space such that the Hardy-Little\-wood
maximal operator $M$ is bounded on $X(\R)$ and on its associate
space $X'(\R)$. 

\begin{enumerate}
{
\item[(a)] If $b\in C_0(\R)\cap V(\R)$ and $\{\upsilon_n\}_{n=1}^\infty$
is the sequence of functions in $C_c^\infty(\mathbb{R})$
defined in Lemma~\ref{le:approximating-vanishing}, then}
\[
{
\lim_{n\to\infty}\|b-\upsilon_n b\|_{\cM_{X(\R)}}=0.
}
\]

\item[{(b)}]
If $b\in C_X(\dR)$ is such that $b(\infty)=0$, then
there exists a sequence $\{b_n\}_{n=1}^\infty$ of functions in
$C_0(\R)\cap V(\R)$ such that
\[
\lim_{n\to\infty}\|b_n-b\|_{\cM_{X(\R)}}=0.
\]
\end{enumerate}
\end{lemma}
\begin{proof}
{
Part (a) follows from Lemma~\ref{le:approximating-vanishing}(b) and
inequality \eqref{eq:Stechkin}.
}

{(b)}
It follows from the definition of $C_X(\dR)$ that there exists a sequence
$\{d_n\}_{n=1}^\infty$ in $C(\dR)\cap V(\R)$ such that
\begin{equation}\label{eq:approximating-vanishing-multipliers-1}
\lim_{n\to\infty}\|d_n-b\|_{\cM_{X(\R)}}=0.
\end{equation}
Take $b_n:=d_n-d_n(\infty)$. Then $b_n\in C_0(\R)\cap V(\R)$. It follows
\eqref{eq:approximating-vanishing-multipliers-1} and 
Theorem~\ref{th:continuous-embedding} that $\{d_n\}_{n=1}^\infty$ converges
uniformly to $b$ on $\R$. In particular,
\begin{equation}\label{eq:approximating-vanishing-multipliers-2}
\lim_{n\to\infty}d_n(\infty)=b(\infty)={0.}
\end{equation}
Combining \eqref{eq:approximating-vanishing-multipliers-1} and
\eqref{eq:approximating-vanishing-multipliers-2}, we see that
\begin{align*}
\lim_{n\to\infty}\|b_n-b\|_{\cM_{X(\R)}}
&=
\lim_{n\to\infty}\|d_n-d_n(\infty)-b\|_{\cM_{X(\R)}}
\\
&\le
\lim_{n\to\infty}\|d_n-b\|_{\cM_{X(\R)}}
+\lim_{n\to\infty}|d_n(\infty)|=0,
\end{align*}
which completes the proof.
\end{proof}
\section{Commutativity of the algebra \boldmath{$\cA_{X(\R)}^\pi$}}
\label{sec:commutativity}
\subsection{Compactness of convolution operators from a subspace of compactly 
supported functions of $L^1(\R)$ to a subspace of compactly
supported functions of $C(\R)$}
Let $C^k(\mathbb{R})$, $k = 0, 1, 2, \dots$ be the space of functions with 
continuous bounded derivatives of all orders up to $k$,
$C(\mathbb{R}) = C^0(\mathbb{R})$. For any space of functions $Y(\mathbb{R})$ 
and any $R > 0$, let $Y_{[R]}(\mathbb{R})$ denote the subspace of 
$Y(\mathbb{R})$ consisting of functions with supports in $[-R, R]$.
As usual, the support of a function $f:\R\to\C$ will be denoted 
{by} $\operatorname{supp}f$.
\begin{lemma}\label{le:compact-convolution}
Suppose that $R_1,R_2>0$. If $k\in C^1(\R)$ is such that 
$\operatorname{supp}k\subset[-R_1,R_1]$, then the convolution operator with 
the kernel $k$ defined by
\begin{equation}\label{eq:compact-convolution}
(Kf)(x):=(k*f)(x)=\int_\R k(x-y)f(y)\,dy,
\quad x\in\R,
\end{equation}
is compact from the space $L^1_{[R_2]}(\R)$ to the space $C_{[R_1+R_2]}(\R)$.
\end{lemma}
\begin{proof}
It follows from \cite[Propositions~4.18 and 4.20]{B11} that the operator
$K$ is bounded from the space $L^1_{[R_2]}(\R)$ to the space 
$C^1_{[R_1+R_2]}(\R)$. Further, by the Arzel\`a-Ascoli theorem
(see, e.g., \cite[Theorems~2.2.12 and 2.5.10]{PKJF13}), the space 
$C_{[R_1+R_2]}^1(\R)$ is compactly embedded into the space 
$C_{[R_1+R_2]}(\R)$, which completes the proof.
\end{proof}
\subsection{Compactness of products of Fourier convolution
operators and multiplication operators}
The main step in the proof of Theorem~\ref{th:algebra-Api} consists
of proving the following. 
\begin{theorem}\label{th:compactness-products}
Let $X(\R)$ be a separable Banach function space such that the Hardy-Littlewood 
maximal operator $M$ is bounded on  $X(\R)$ and on its associate 
space $X'(\R)$. If $a \in C(\dR)$ and $b \in C_X(\dR)$ are such 
that $a(\infty)= b(\infty)=0$, then 
\[
aW^0(b), W^0(b)aI \in \mathcal{K}(X(\mathbb{R})).
\]
\end{theorem}
\begin{proof}
A part of the proof is quite standard (see, e.g., 
\cite[Theorem 5.3.1(i)]{RSS11}).
It follows from 
Lemma~\ref{le:approximating-vanishing-multipliers}{(b)} 
that there 
exists a sequence $\{b_n\}_{n=1}^\infty$ of functions in $C_0(\R)\cap V(\R)$
such that $\|b_n-b\|_{\cM_{X(\R)}}\to 0$ as $n\to\infty$. Then
\[
\|aW^0(b)-aW^0(b_n)\|_{\cB(X(\R))}\to 0,
\quad
\|W^0(b)aI-W^0(b_n)aI\|_{\cB(X(\R))}\to 0
\]
as $n\to\infty$. So, we can assume without loss of generality that 
$b\in C_0(\R)\cap V(\R)$.
Let $\{\upsilon_n\}_{n=1}^\infty$ be the sequence of functions 
in $C_c^\infty(\R)$ as in Lemma~\ref{le:approximating-vanishing}.
{Since}
\begin{align*}
&
\|aW^0(b)-\upsilon_n aW^0(\upsilon_n b)\|_{\cB(X(\R))}
\\
&\quad
\le 
\|(a-\upsilon_n a)W^0(b)\|_{\cB(X(\R))}
+
\|\upsilon_n a W^0(b-\upsilon_n b)\|_{\cB(X(\R))}
\\
&\quad
\le 
\|a-\upsilon_n a\|_{L^\infty(\R)}\|b\|_{\cM_{X(\R)}}
+ 
{\|a\|_{L^\infty(\R)}\|b-\upsilon_n b\|_{\cM_{X(\R)}},}
\end{align*}
{
Lemmas~\ref{le:approximating-vanishing}(a) 
and~\ref{le:approximating-vanishing-multipliers}(a)
} 
imply that
\[
\lim_{n\to\infty}\|aW^0(b)-\upsilon_n aW^0(\upsilon_n b)\|_{\cB(X(\R))}=0.
\]
Analogously we can show that
\[
\lim_{n\to\infty}\|W^0(b)aI-W^0(\upsilon_n b)\upsilon_n a I\|_{\cB(X(\R))}=0.
\]
{
Taking into account that $\upsilon_n\in C_c^\infty(\mathbb{R})$
and
\[
\upsilon_n aW^0(\upsilon_n b)
=
a(\upsilon_n W^0(\upsilon_n))W^0(b),
\
W^0(\upsilon_n b)\upsilon_naI
=
W^0(b)(W^0(\upsilon_n)\upsilon_n)aI,
\]
it is enough to prove that 
$a_0W^0(b_0)$ and $W^0(b_0)a_0I$ are compact operators
for all $a_0,b_0\in C_c^\infty(\mathbb{R})$.
}

Since $F^{-1}b_0 \in \mathcal{S}(\mathbb{R})$, it is easy to see that 
$\upsilon_n F^{-1}b_0 \to F^{-1}b_0$ in $\mathcal{S}(\mathbb{R})$ as
$n \to \infty$. Then $b_n := F\left(\upsilon_n F^{-1}b_0\right) \to b_0$ 
in $\mathcal{S}(\mathbb{R})$ as $n\to\infty$. It is easy to see that the 
convergence in $\mathcal{S}(\R)$ implies the convergence in $V(\R)$.
Therefore $\|b_n - b_0\|_{V(\mathbb{R})} \to 0$ as $n \to \infty$. 
It follows from the Stechkin type inequality \eqref{eq:Stechkin} that
\begin{align*}
\lim_{n\to\infty}
&
\|a_0W^0(b_n)-a_0W^0(b_0)\|_{\cB(X(\R))}
\\
&
\le 
c_X\|a_0\|_{L^\infty(\R)}\lim_{n\to\infty}\|b_n-b_0\|_{V(\R)}
=0,
\\
\lim_{n\to\infty}
&
\|W^0(b_n)a_0I-W^0(b_0)a_0I\|_{\cB(X(\R))}
\\
&
\le 
c_X\|a_0\|_{L^\infty(\R)}\lim_{n\to\infty}\|b_n-b_0\|_{V(\R)}
=0.
\end{align*}
Thus, it is sufficient to prove that
$a_0W^0(b_n), W^0(b_n)a_0I \in \mathcal{K}(X(\mathbb{R}))$ for all $n\in\N$. 
Let 
$k_n:=F^{-1}b_n=\upsilon_n F^{-1}b_0\in C_c^\infty(\R)$.
It follows from the convolution theorem for the inverse Fourier
transform (see, e.g., \cite[Proposition~2.2.11, statement (12)]{G14})
that
for all $n\in\N$ and $f\in C_c^\infty(\R)$,
\begin{align}\label{eq:compactness-products-1}
W^0(b_n)f
&=
F^{-1}(b_nFf)=(F^{-1}b_n)*F^{-1}(Ff)
\nonumber\\
&=
(F^{-1}b_n)*f=k_n*f=:K_nf,
\end{align}
where $K_n$ is the convolution operator with the kernel $k_n$ defined by
\eqref{eq:compact-convolution}.
In view of Lemma~\ref{le:density}, equality \eqref{eq:compactness-products-1}
remains valid for all $f\in X(\R)$. 

Take $R_1,R_2>0$ such that $\operatorname{supp}k_n\subset[-R_1,{R_1}]$
and $\operatorname{supp}a_0\subset[-R_2,R_2]$. Equality 
\eqref{eq:compactness-products-1} implies that
\[
a_0W^0(b_n)=a_0K_n\chi_{[-R_1-R_2,R_1+R_2]}I.
\]
It follows from Axiom (A5) that there exists 
$C_{[-R_1-R_2,R_1+R_2]}\in(0,\infty)$ such that for all $f\in X(\R)$,
\[
\int_{-R_1-R_2}^{R_1+R_2}|f(x)|dx
\le 
C_{[-R_1-R_2,R_1+R_2]} \|f\|_{X(\R)},
\]
which means that the operator $\chi_{[-R_1-R_2,R_1+R_2]}I$ is bounded
from the space $X(\R)$ to the space $L^1_{[R_1+R_2]}(\R)$.
By Lemma~\ref{le:compact-convolution}, the operator $K_n$ is compact
from the space $L^1_{[R_1+R_2]}(\R)$ to the space $C_{[2R_1+R_2]}(\R)$.
It follows from Axiom (A2) that the operator 
$a_0I: C_{[2R_1+R_2]}(\R)\to X(\R)$ is bounded. Thus, for every $n\in\N$,
the operator $a_0W^0(b_n):X(\R)\to X(\R)$ is compact as the composition
of the bounded operator $\chi_{[-R_1-R_2,R_1+R_2]}I:X(\R)\to 
L^1_{[R_1+R_2]}(\R)$, the compact operator 
${K_n}:L_{[R_1+R_2]}^1(\R)\to C_{[2R_1+R_2]}(\R)$, 
and the bounded operator $aI:C_{[2R_1+R_2]}(\R)\to
X(\R)$.

Similarly, for every $n\in\N$, the operator $W^0(b_n) a_0I:X(\R)\to X(\R)$
is compact as the composition of the bounded operator
$a_0I:X(\R)\to L_{[R_2]}^1(\R)$, the compact operator
$K_n:L_{[R_2]}^1(\R)\to C_{[R_1+R_2]}(\R)$, and the bounded
operator $I:C_{[R_1+R_2]}(\R)\to X(\R)$.
\end{proof}
\subsection{Compactness of commutators of Fourier convolution
operators and multiplication operators} 
The previous theorem implies the following.
\begin{corollary}\label{co:compactness-commutators}
Let $X(\R)$ be a separable Banach function space such that the Hardy-Littlewood 
maximal operator $M$ is bounded on $X(\R)$ and on its associate 
space $X'(\R)$. If $a \in C(\dR)$ and $b \in C_X(\dR)$, then 
\[
[aI, W^0(b)] := aW^0(b) - W^0(b)aI \in \mathcal{K}(X(\mathbb{R})).
\]
\end{corollary}
\begin{proof}
Since $a = a(\infty) + \widetilde{a}$ and $b = b(\infty) + \widetilde{b}$, 
where $\widetilde{a} \in C(\dR)$, $\widetilde{b} \in C_X(\dR)$, and 
$\widetilde{a}(\infty) = 0 = \widetilde{b}(\infty)$,
Theorem~\ref{th:compactness-products} implies that
\begin{align*}
[aI, W^0(b)] 
&= 
(a(\infty) + \widetilde{a}) (b(\infty) + W^0(\widetilde{b})) 
- 
(b(\infty) + W^0(\widetilde{b}))(a(\infty) + \widetilde{a}) I 
\\
& = 
\widetilde{a} W^0(\widetilde{b}) - W^0(\widetilde{b}) \widetilde{a}I 
\in \mathcal{K}(X(\mathbb{R})).
\qedhere
\end{align*}
\end{proof}
\subsection{Proof of Theorem~\ref{th:algebra-Api}}
Since a Banach function space $X(\R)$ is reflexive if and only if the space 
$X(\R)$ and its associate space $X'(\R)$ are separable 
(see 
{
\cite[Chap.~1, \S2, Theorem 4 and \S3, Corollary 1 to Theorem 7]{L55} or} 
\cite[Chap.~1, Corollaries 4.4 and 5.6]{BS88}), 
Theorem~\ref{th:algebra-Api} follows from Theorem~\ref{th:FKK} and
Corollary~\ref{co:compactness-commutators}.
\qed
\section{Proof of the main result}
\label{sec:main-proof}
\subsection{Estimate for the norm of a product of multiplication operators and
a Fourier convolution operator}
For $n\in\N_0:=\N\cup\{0\}$, let
\[
\ell_n(x) := \frac{\log^n (1 + |x|)}{1 + |x|}, \quad x \in \mathbb{R}.
\]
\begin{lemma}\label{le:ell-n}
If $Y(\R)$ is a Banach function space such that the Hardy-Littlewood
maximal operator $M$ is bounded on it, then $\ell_n \in Y(\mathbb{R})$
for all $n\in\N_0$.
\end{lemma}
\begin{proof}
Since $\chi_{[-1, 1]} \in Y(\mathbb{R})$ by Axiom (A4), 
$M\chi_{[-1, 1]} \in Y(\mathbb{R})$. It is easy to see
that $0 \le \ell_0 \le M\chi_{[-1, 1]}$ (see \cite[Example 2.1.4]{G14}). 
Hence $\ell_0 \in Y(\mathbb{R})$ in view of Axiom (A2).

Now let $k\in\N_0$. It follows from the definition of the 
Hardy-Littlewood maximal operator that for $x\ne 0$,
\[
(M\ell_k)(x)
\ge 
\left\{\begin{array}{lll}
\displaystyle
\frac{1}{x+\eps}\int_0^{x+\eps}\frac{\log^k(1+|t|)}{1+|t|}dt
&\mbox{if}& x,\eps>0,
\\[4mm]
\displaystyle
\frac{1}{-x-\eps}\int_{x+\eps}^0\frac{\log^k(1+|t|)}{1+|t|}dt
&\mbox{if}& x,\eps<0.
\end{array}\right.
\]
Passing to the limit as $\eps\to 0^\pm$, we obtain for $x\ne 0$,
\begin{align*}
(M\ell_k)(x)
&\ge 
\left\{\begin{array}{lll}
\displaystyle
\frac{1}{x}\int_0^{x}\frac{\log^k(1+|t|)}{1+|t|}dt
&\mbox{if}& x>0,
\\[4mm]
\displaystyle
\frac{1}{-x}\int_{x}^0\frac{\log^k(1+|t|)}{1+|t|}dt
&\mbox{if}& x<0
\end{array}\right.
=
\frac{1}{|x|}\int_0^{|x|}\frac{\log^k(1+t)}{1+t}dt
\\
&=
\frac{1}{(k+1)|x|}\log^{k+1}(1+|x|)
\ge
\frac{1}{k+1}\ell_{k+1}(x).
\end{align*}
So
\[
0\le\ell_{k+1}\le(k+1)M\ell_k,\quad k\in\N_0,
\]
and one gets by induction that $\ell_n\in Y(\R)$ for all $n\in\N_0$.
\end{proof}
For $R>0$, let $\chi_{\{R\}} := \chi_{\mathbb{R}\setminus[-R, R]}$.
\begin{theorem}\label{th:key-estimate}
Let $X(\mathbb{R})$ be a separable Banach function space such 
that the Hardy-Littlewood maximal operator is bounded on $X(\R)$ 
and on its associate space $X'(\mathbb{R})$. 
Let $a \in C_c(\R)$ and $b \in C_c(\mathbb{R})\cap V(\mathbb{R})$. 
Then for every $n \in \mathbb{N}_0$, there exists a constant 
$c_n(a,b)\in (0,\infty)$ depending only on $a,b$ and $n$,
such that for all $R>0$,
\begin{equation}\label{eq:key-estimate-0}
\|aW^0(b)\chi_{\{R\}}I\|_{\cB(X(\R))} \le \frac{c_n(a,b)}{\log^n(R + 2)}.
\end{equation}
\end{theorem}
\begin{proof}
Since $b\in C_c(\R)\subset L^1(\R)$, it follows from the convolution theorem
for the inverse Fourier transform (see, e.g., \cite[Theorem~11.66]{A20})
that for $f\in C_c^\infty(\R)$,
\begin{equation}\label{eq:key-estimate-1}
W^0(b)f=F^{-1}(b\cdot Ff)=(F^{-1}b)*F^{-1}(Ff)=:k*f,
\end{equation}
where $k:=F^{-1}b$. In view of Lemma~\ref{le:density}, formula
\eqref{eq:key-estimate-1} remains valid for all $f\in X(\R)$.
Since $b\in V(\R)$, using integration by parts, similarly to the proof
of \cite[Chap.~I, Theorem~4.5]{K76}, we get for $x\in\R$,
\[
k(x)
=
(F^{-1}b)(x)
=
\frac{1}{2\pi}\int_\R e^{-ix\xi}b(\xi)\,d\xi
=
\frac{1}{2\pi ix}\int_\R e^{-ix\xi}db(\xi),
\]
and hence
\begin{equation}\label{eq:key-estimate-2}
|k(x)|\le \frac{V(b)}{2\pi|x|},
\quad
x\in\R.
\end{equation}
Take $R_1>0$ such that $\operatorname{supp}a\subset[-R_1,R_1]$. If
$x\in[-R_1,R_1]$ and $|y|>R\ge\max\{2R_1,1\}$, then
\begin{equation}\label{eq:key-estimate-3}
|x-y|\ge|y|-|x|\ge|y|-R_1\ge|y|-\frac{|y|}{2}=\frac{|y|}{2}\ge\frac{|y|+1}{4}
\end{equation}
and
\begin{equation}\label{eq:key-estimate-4}
\log(R+1)\ge\frac{1}{2}\log(R+2).
\end{equation}
Combining \eqref{eq:key-estimate-2}--\eqref{eq:key-estimate-4} and
taking into account the definition of $\ell_n$, we get for every
$x\in[-R_1,R_1]$, $R\ge\max\{2R_1,1\}$, and $n\in\N_0$,
\begin{align}
|k(x-y)|\chi_{\{R\}}(y)
&\le 
\frac{V(b)}{2\pi|x-y|}\chi_{\{R\}}(y)
\le 
\frac{2V(b)}{\pi(1+|y|)}\chi_{\{R\}}(y)
\nonumber\\
&\le 
\frac{2V(b)}{\pi\log^n(R+1)}\ell_n(y)
\le 
\frac{2^{n+1}V(b)}{\pi\log^n(R+2)}\ell_n(y).
\label{eq:key-estimate-5}
\end{align}
It follows from \eqref{eq:key-estimate-5}, Lemma~\ref{le:ell-n} and
H\"older's inequality for Banach function spaces 
(see {\cite[Chap.~1, \S1, Lemma 2]{L55} or}
\cite[Chap.~1, Theorem~2.4]{BS88}) that for $x\in[-R_1,R_1]$, 
$R\ge\max\{2R_1,1\}$, $n\in\N_0$ and $f\in X(\R)$,
\begin{align}
|k*(\chi_{\{R\}}f)(x)|
&=
\left|\int_\R k(x-y)\chi_{\{R\}}(y)f(y)\,dy\right|
\nonumber\\
&\le 
\|k(x-\cdot)\chi_{\{R\}}\|_{X'(\R)}\|f\|_{X(\R)}
\nonumber\\
&\le 
\frac{2^{n+1}V(b)\|\ell_n\|_{X'(\R)}\|f\|_{X(\R)}}{\pi\log^n(R+2)}.
\label{eq:key-estimate-6}
\end{align}
It follows from Axiom (A4) that $\chi_{[-R_1,R_1]}\in X(\R)$.
Since $\operatorname{supp}a\subset[-R_1,R_1]$, in view of Axiom (A2), 
equality \eqref{eq:key-estimate-1} and inequality \eqref{eq:key-estimate-6},
we obtain for $R\ge\max\{2R_1,1\}$, $f\in X(\R)$ and $n\in\N_0$,
\begin{align}
\|aW^0(b)\chi_{\{R\}}f\|_{X(\R)}
&\le 
\|a\|_{L^\infty(\R)}\|\chi_{[-R_1,R_1]}\|_{X(\R)}
\operatornamewithlimits{ess\,sup}_{x\in[-R_1,R_1]}|k*(\chi_{\{R\}}f)(x)|
\nonumber\\
&\le 
\frac{2^{n+1}{\|a\|_{L^\infty(\R)}}
V(b)\|\ell_n\|_{X'(\R)}\|\chi_{[-R_1,R_1]}\|_{X(\R)}}
{\pi\log^n(R+2)}
\|f\|_{X(\R)}.
\label{eq:key-estimate-7}
\end{align}
If $R\in(0,\max\{2R_1,1\})$, then
$\log(R+2)\le\log(2+\max\{2R_1,1\})$ and
\begin{align}
\|aW^0(b)\chi_{\{R\}}I\|_{\cB(X(\R))}
&\le 
\|aW^0(b)\|_{\cB(X(\R))}
\nonumber\\
&\le 
\frac{\log^n(2+\max\{2R_1,1\})\|aW^0(b)\|_{\cB(X(\R))}}{\log^n(R+2)}.
\label{eq:key-estimate-8}
\end{align}
It follows from \eqref{eq:key-estimate-7} and \eqref{eq:key-estimate-8}
that \eqref{eq:key-estimate-0} is fulfilled with
\begin{align*}
c_n(a,b):=\max\bigg\{ &
\frac{2^{n+1}}{\pi}{\|a\|_{L^\infty(\R)}}
V(b)\|\ell_n\|_{X'(\R)}\|\chi_{[-R_1,R_1]}\|_{X(\R)},
\\
&
\log^n(2+\max\{2R_1,1\})\|aW^0(b)\|_{\cB(X(\R))}\bigg\},
\end{align*}
which completes the proof.
\end{proof}
\subsection{Sufficient condition on the space \boldmath{$X(\R)$} implying that the \\
algebra \boldmath{$\cA_{X(\R)}$} does not contain all rank one operators}
Now we prove a conditional statement, which will lead 
to the proof of Theorem~\ref{th:main}.
\begin{theorem}\label{th:main-key-step}
Let $X(\mathbb{R})$ be a separable non-reflexive Banach function space such 
that the Hardy-Littlewood maximal operator is bounded on $X(\R)$ 
and on its associate space $X'(\mathbb{R})$. Suppose that there exist a 
function $g \in X'(\mathbb{R})$ and a constant $\delta > 0$ such that 
$\|\chi_{\{R\}}g\|_{X'(\mathbb{R})} \ge \delta$ for all $R > 0$. Then for 
any function $h \in X(\mathbb{R})\setminus\{0\}$, the rank one operator 
$T_{g,h} \in \mathcal{B}(X(\mathbb{R}))$, defined by
\[
(T_{g,h}f)(x) := h(x) \int_\R g(y) f(y)\, dy ,
\]
does not belong to the algebra $\mathcal{A}_{X(\mathbb{R})}$.
\end{theorem}
\begin{proof}
Fix $h\in X(\R)\setminus\{0\}$. Suppose the contrary: $T_{g,h}\in\cA_{X(\R)}$.
Fix $\eps>0$. By the definition of the algebra $\cA_{X(\R)}$ there exist
numbers $N,M\in\N$ and operators 
\[
A_{ij}\in\{aI,W^0(b)\ : \ a\in C(\dR),\ b\in C_X(\dR)\}
\]
for $i\in\{1,\dots,N\}$, $j\in\{1,\dots,M\}$ such that
\begin{equation}\label{eq:main-key-step-1}
\left\|
T_{g,h}-\sum_{i=1}^N A_{i1}\dots A_{iM}
\right\|_{\cB(X(\R))}<\frac{\eps}{6}.
\end{equation}
Put
\begin{equation}\label{eq:main-key-step-2}
L:=2\max\big\{
\|A_{ij}\|_{\cB(X(\R))}
\ :\
i\in\{1,\dots,{N}\},\
j\in\{1,\dots,M\}{\big\} .}
\end{equation}
Let $b_1,\dots,b_r\in C_X(\dR)$ be such that for $k\in\{1,\dots,r\}$,
\[
W^0(b_k)\in
\big\{A_{ij} \ :\
i\in\{1,\dots,{N}\},\
j\in\{1,\dots,M\}\big\} 
\setminus
\big\{aI\ : \ a\in C(\dR)\big\}
\]
and $a_1,\dots, a_s\in C(\dR)$ be such that for $l\in\{1,\dots,s\}$,
\[
a_lI
\in
\big\{A_{ij}  :
i\in\{1,\dots,{N}\},\
j\in\{1,\dots,M\}\big\} 
\setminus
\big\{W^0(b_k) :  k\in\{1,\dots,r\}\big\}.
\]
It follows from the definition of the algebra $C_X(\dR)$ that for each
$k\in\{1,\dots,r\}$ there exists a function $c_k\in C(\dR)\cap V(\R)$
such that
\begin{align}
\|W^0(b_k)-W^0(c_k)\|_{\cB(X(\R))}
&=
\|b_k-c_k\|_{\cM_{X(\R)}}
\nonumber\\
&<
\min\left\{\frac{\eps}{6NML^{M-1}},\frac{L}{4}\right\}.
\label{eq:main-key-step-3}
\end{align}
Further, in view of 
{Lemma~\ref{le:approximating-vanishing-multipliers}(a),}
there exists
a function $\widetilde{b}_k\in C_c(\R)\cap V(\R)$ such that
\begin{align}
\|W^0(c_k)-c_k(\infty)I-W^0(\widetilde{b}_k)\|_{\cB(X(\R))}
&=
\|c_k-c_k(\infty)-\widetilde{b}_k\|_{\cM_{{X(\R)}}}
\nonumber\\
&
{
<\min\left\{\frac{\eps}{6NML^{M-1}},\frac{L}{4}\right\}.
}
\label{eq:main-key-step-4}
\end{align}
Combining \eqref{eq:main-key-step-3} and \eqref{eq:main-key-step-4},
we get
\[
\|W^0(b_k)-c_k(\infty)I-W^0(\widetilde{b}_k)\|_{\cB(X(\R))}
<
\min\left\{\frac{\eps}{3NML^{M-1}},\frac{L}{2}\right\}.
\]
Analogously, by Lemma~\ref{le:approximating-vanishing}(a), for every
$l\in\{1,\dots,s\}$, there exists $\widetilde{a}_l\in C_c(\R)$
such that
\begin{align*}
\|a_lI-a_l(\infty)I-\widetilde{a}_lI\|_{\cB(X(\R))}
&\le 
\|a_l-a_l(\infty)-\widetilde{a}_l\|_{L^\infty(\R)}
\\
&<
\min\left\{\frac{\eps}{3NML^{M-1}},\frac{L}{2}\right\}.
\end{align*}
We have shown that for every $i\in\{1,\dots,N\}$ and $j\in\{1,\dots,{M}\}$
there exists an operator
\begin{equation}\label{eq:main-key-step-5}
B_{ij}\in\big\{cI+\widetilde{a}I,cI+W^0(\widetilde{b})\ :\
c\in\C,\ \widetilde{a}\in C_c(\R),\ \widetilde{b}\in C_c(\R)\cap V(\R)\big\}
\end{equation}
such that
\[
\|A_{ij}-B_{ij}\|_{\cB(X(\R))}
<
\min\left\{\frac{\eps}{3NML^{M-1}},\frac{L}{2}\right\}.
\]
Then, taking into account \eqref{eq:main-key-step-2}, we get
\begin{align*}
&
\left\|
\sum_{i=1}^N A_{i1}\dots A_{iM}
-
\sum_{i=1}^N B_{i1}\dots B_{iM}
\right\|_{\cB(X(\R))}
\\
&\quad=
\left\|
\sum_{i=1}^N\sum_{j=1}^M 
A_{i1}\dots A_{i,j-1}
(A_{ij}-B_{ij})B_{i,j+1}\dots B_{iM}
\right\|_{\cB(X(\R))}
\\
&\quad\le
\sum_{i=1}^N\sum_{j=1}^M
\left(\prod_{k=1}^{j-1}\|A_{ik}\|_{\cB(X(\R))}\right)
\|A_{ij}-B_{ij}\|_{\cB(X(\R))} 
\left(\prod_{l=j+1}^{M}\|B_{il}\|_{\cB(X(\R))}\right)
\end{align*}
\begin{align}
<
\sum_{i=1}^N\sum_{j=1}^M
\left(\frac{L}{2}\right)^{j-1}
\frac{\eps}{3NML^{M-1}}
\left(\frac{L}{2}+\frac{L}{2}\right)^{M-j}
<
\sum_{i=1}^N\sum_{j=1}^M
\frac{\eps}{3NM}
=
\frac{\eps}{3}.
\label{eq:main-key-step-6}
\end{align}
It follows from \eqref{eq:main-key-step-1} and
\eqref{eq:main-key-step-6} that
\begin{equation}\label{eq:main-key-step-7}
\|T_{g,h}-T_\eps\|_{\cB(X(\R))}
<
\frac{\eps}{6}+\frac{\eps}{3}
=
\frac{\eps}{2},
\end{equation}
where
\[
T_\eps:=\sum_{i=1}^N B_{i1}\dots B_{iM}.
\]
Taking into account \eqref{eq:main-key-step-5}, we can rearrange
terms and write the operator $T_\eps$ in the form
\begin{equation}\label{eq:main-key-step-8}
T_\eps=cI+W^0(\widetilde{b}_0)
+
\sum_{i=1}^p D_{1,i}\widetilde{a}_{1,i}I
+
\sum_{j=1}^t D_{2,j}\widetilde{a}_{2,j}W^0(\widetilde{b}_j),
\end{equation}
where $c\in\C$, $\widetilde{b}_j\in C_c(\R)\cap V(\R)$
for $j\in\{0,\dots,t\}$, 
$\widetilde{a}_{1,i},\widetilde{a}_{2,j}\in C_c(\R)$
and $D_{1,i},D_{2,j}$ are some operators in
$\cA_{X(\R)}\setminus\{0\}$
for $i\in\{1,\dots,p\}$ and $j\in\{1,\dots,t\}$.

Since the space $X(\R)$ is separable, it follows from 
{
\cite[Chap.~1, \S2, Definition 1 and \S3, Corollary 1 to Theorem 7]{L55} 
(or \cite[Chap.~1, Definition~3.1 and Corollary~5.6]{BS88})} that there
exists $R_1>0$ such that 
$\|\chi_{\{R_1\}}h\|_{X(\R)}\le\frac{1}{2}\|h\|_{X(\R)}$. Then
\begin{equation}\label{eq:main-key-step-9}
\|\chi_{R_1}h\|_{X(\R)}
\ge 
\|h\|_{X(\R)}-\|\chi_{\{R_1\}}h\|_{X(\R)}
\ge 
\frac{1}{2}\|h\|_{X(\R)},
\end{equation}
where 
\[
\chi_{R_1}:=1-\chi_{\{R_1\}}=\chi_{[-R_1,R_1]}.
\]
Since $\widetilde{a}_{1,i}\in C_c(\R)$ for $i=1,\dots,p$, there
exists $R_2>R_1$ such that for $R\ge R_2$,
\begin{equation}\label{eq:main-key-step-10}
\chi_{R_1}(cI)\chi_{\{R\}}I
+
\chi_{R_1}\sum_{i=1}^p D_{1,i}\widetilde{a}_{1,i}\chi_{\{R\}}I=0.
\end{equation}
Let $\widetilde{a}_0\in C_c(\R)$ be such that $\widetilde{a}_0=1$
for $x\in[-R_1,R_1]$. Then
\[
\chi_{R_1}W^0(\widetilde{b}_0)=\chi_{R_1}\widetilde{a}_0 W^0(\widetilde{b}_0).
\]
It follows from Theorem~\ref{th:key-estimate} that there exists 
$R_0>R_2$ such that for all $R\ge R_0$ and $j\in\{1,\dots,t\}$,
\begin{align}
&
\|\chi_{R_1}\widetilde{a}_0W^0(\widetilde{b}_0)\chi_{\{R\}}I\|_{\cB(X(\R))}
\le 
\|\widetilde{a}_0W^0(\widetilde{b}_0)\chi_{\{R\}}I\|_{\cB(X(\R))}
<
\frac{\eps}{2(t+1)},
\label{eq:main-key-step-11}
\\
&
\|\widetilde{a}_{2,j}W^0(\widetilde{b}_j)\chi_{\{R\}}I\|_{\cB(X(\R))}
<
\frac{\eps}{2(t+1)\|D_{2,j}\|_{\cB(X(\R))}}.
\label{eq:main-key-step-12}
\end{align}
Combining \eqref{eq:main-key-step-8} and
\eqref{eq:main-key-step-10}--\eqref{eq:main-key-step-12}, we see that
for all $R\ge R_0$,
\begin{equation}\label{eq:main-key-step-13}
\|\chi_{R_1}T_\eps\chi_{\{R\}}I\|_{\cB(X(\R))}
<\frac{\eps}{2(t+1)}+\sum_{j=1}^t\frac{\eps}{2(t+1)}=\frac{\eps}{2}.
\end{equation}
It follows from \eqref{eq:main-key-step-7} and \eqref{eq:main-key-step-13}
that for all $R\ge R_0$,
\begin{align}
\|\chi_{R_1} T_{g,h}\chi_{\{R\}}I\|_{\cB(X(\R))}
&
\le
\|\chi_{R_1} (T_{g,h}-T_\eps)\chi_{\{R\}}I\|_{\cB(X(\R))}
\nonumber\\
&\quad +
\|\chi_{R_1} T_\eps\chi_{\{R\}}I\|_{\cB(X(\R))}
\nonumber\\
&\le
\|T_{g,h}-T_\eps\|_{\cB(X(\R))}+\frac{\eps}{2}<\eps.
\label{eq:main-key-step-14}
\end{align}

On the other hand, in view of \cite[Chap.~1, Lemma~2.8]{BS88} 
{
(see also \cite[Chap.~1, \S1, Remark (2) after Theorem 2]{L55}),}
we have
\begin{align*}
&
\|\chi_{R_1} T_{g,h}\chi_{\{R\}}I\|_{\cB(X(\R))}
\\
&\quad=
\sup\left\{
\left\|\chi_{R_1}h\int_\R g(y)\chi_{\{R\}}(y)f(y)\,dy\right\|_{X(\R)}
\ :\
f\in X(\R),\ \|f\|_{X(\R)}\le 1
\right\}
\\
&\quad=
\sup\left\{
\left|\int_\R g(y)\chi_{\{R\}}(y)f(y)\,dy\right|
\|\chi_{R_1}h\|_{X(\R)}
\ :\
f\in X(\R),\ \|f\|_{X(\R)}\le 1
\right\}
\\
&\quad=\|\chi_{R_1}h\|_{X(\R)}
\sup\left\{
\left|\int_\R g(y)\chi_{\{R\}}(y)f(y)\,dy\right|
\ :\
f\in X(\R),\ \|f\|_{X(\R)}\le 1
\right\}
\\
&\quad=
\|\chi_{R_1}h\|_{X(\R)}\|g\chi_{\{R\}}\|_{X'(\R)}.
\end{align*}
This equality, inequality \eqref{eq:main-key-step-9} and inequality
$\|\chi_{\{R\}}g\|_{X'(\R)}\ge\delta$ imply that
\begin{equation}\label{eq:main-key-step-15}
\|\chi_{R_1}T_{g,h}\chi_{\{R\}}I\|_{\cB(X(\R))}
\ge\frac{\delta}{2}\|h\|_{X(\R)}.
\end{equation}
Inequalities \eqref{eq:main-key-step-14} and \eqref{eq:main-key-step-15}
yield a contradiction for $\eps\le\frac{\delta}{2}\|h\|_{X(\R)}$.
\end{proof}
\begin{remark}
Note that a Banach function spaces $X(\R)$ is reflexive if and only if
$X(\R)$ and its associate space $X'(\R)$ are separable
{
(see \cite[Chap.~1, \S2, Theorem 4 and \S3, Corollary 1 to Theorem 7]{L55} or
\cite[Chap.~1, Corollaries~4.4 and 5.6]{BS88}).} 
In turn, if $X'(\R)$ is separable, then for any $g\in X'(\R)$ one has 
$\|\chi_{\{R\}}g\|_{X'(\R)}\to 0$ as $R\to\infty$ in view of 
{
\cite[Chap.~1, \S2, Definition 1 and \S3, Corollary 1 to Theorem 7]{L55}
(or \cite[Chap.~1, Definition~3.1 and Corollary~5.6]{BS88}).}
\end{remark}
To complete the proof of Theorem~\ref{th:main}, we have to show that
there exists a separable non-reflexive Banach function space satisfying the
hypotheses of Theorem~\ref{th:main-key-step}. In the next subsection, we will
show that the classical Lorentz spaces $L^{p,1}(\R)$, $1<p<\infty$,
perfectly fit our needs. 
\subsection{Proof of Theorem~\ref{th:main-Lorentz}}
The space $X(\mathbb{R}) = L^{p, 1}(\mathbb{R})$ is separable and 
\[
\left[L^{p, 1}(\mathbb{R})\right]^* 
= 
\left(L^{p, 1}\right)' (\mathbb{R})
= 
L^{p', \infty}(\mathbb{R}),
\]
where $1/p + 1/p' = 1$ (see \cite[Chap. 1, Corollaries 4.3 and 5.6, 
Chap. 4, Corollary 4.8]{BS88}). 
It is also known that 
\[
L^{p,1}(\R)\subsetneqq [L^{p',\infty}(\R)]^*=[L^{p,1}(\R)]^{**}
\]
(see \cite[p.~83]{C75}). Hence $L^{p,1}(\R)$ is non-reflexive.
The lower and upper Boyd indices of 
$L^{p, 1}(\mathbb{R})$ (resp., of $L^{p', \infty}(\mathbb{R})$) are both 
equal to $1/p$ (resp., to $1/p'$); see \cite[Chap. 4, Theorem 4.6]{BS88}. 
Hence the Hardy-Littlewood maximal operator is bounded on the space 
$X(\mathbb{R})$ and on its associate space $X'(\mathbb{R})$ in view of the 
Lorentz-Shimogaki theorem (see \cite[Chap. 3, Theorem 5.17]{BS88}). Thus, the 
space $L^{p,1}(\R)$ is a separable non-reflexive Banach function space
satisfying condition (a) of Theorem~\ref{th:main}.

Consider the function $g(x)=|x|^{-1/p'}$. Its distribution function is
\[
\mu_g(\lambda)
=
|\{x\in\R:|x|^{-1/p'}>\lambda\}|
=
|\{x\in\R:|x|<\lambda^{-p'}\}|=2\lambda^{-p'}, 
\quad
\lambda\ge 0,
\]
and its non-increasing rearrangement is
\begin{align*}
g^*(t)
&=
\inf\{\lambda\ge 0: 2\lambda^{-p'}\le t\}
=
\inf\{\lambda\ge 0: 2^{1/p'}t^{-1/p'}\le \lambda\}
\\
&=
2^{1/p'}t^{-1/p'},
\quad t\ge 0.
\end{align*}
Then
\[
g^{**}(t)
=
\frac{1}{t}\int_0^t 2^{1/p'}y^{-1/p'}dy
=\frac{2^{1/p'}t^{-1/p'}}{1-1/p'}
=
2^{1/p'}pt^{-1/p'},
\quad t\ge 0
\]
and
\[
\|g\|_{(p',\infty)}=2^{1/p'}p<\infty.
\]
The distribution function of $\chi_{\{R\}}g$ for every $R>0$ is given by
\begin{align*}
\mu_{\chi_{\{R\}}g}(\lambda)
&=
|\{x\in\R:\chi_{\{R\}}(x)g(x)>\lambda\}|
\\
&=
\left\{\begin{array}{lll}
2\lambda^{-p'}-2R & \mbox{if}& 0\le\lambda<R^{-1/p'},
\\
0 &\mbox{if}& \lambda\ge R^{-1/p'}.
\end{array}\right.
\end{align*}
Then
\begin{align*}
(\chi_{\{R\}}g)^*(t)
&=
\inf\{\lambda\ge 0 : 2\lambda^{-p'}-2R\le t\}
=
\inf\left\{\lambda\ge 0:\lambda^{-p'}\le\frac{t}{2}+R\right\}
\\
&=
\inf\left\{\lambda\ge 0:\frac{2}{t+2R}\le\lambda^{p'}\right\}
=
2^{1/p'}(t+2R)^{-1/p'},
\quad
t\ge 0.
\end{align*}
Since $(\chi_{\{R\}}g)^*$ is non-increasing, we have
$(\chi_{\{R\}}g)^{**}\ge (\chi_{\{R\}}g)^*$ and
\begin{align*}
\|\chi_{\{R\}}g\|_{(p',\infty)}
&\ge 
\sup_{0<t<\infty}\left(t^{1/p'}(\chi_{\{R\}}g)^*(t)\right)
\\
&=
2^{1/p'}\sup_{0<t<\infty}
\left(\frac{t}{t+2R}\right)^{1/p'}
=
2^{1/p'}.
\end{align*}
Thus, the conditions of Theorem~\ref{th:main-key-step} are satisfied for
$X(\R)=L^{p,1}(\R)$, $g(x)=|x|^{-1/p'}$ and $\delta=2^{1/p'}$.
The desired result now follows from that theorem.
\qed
\section{Final remarks on algebras of convolution type operators with 
continuous and slowly oscillating data}
\label{sec:SO}
\subsection{Algebra \boldmath{$C_X^0(\dR)$} of continuous Fourier multipliers}
{Let} 
$\C$ stand for the constant complex-valued functions on $\R$. Notice that 
$C(\dR)$ decomposes into the direct sum $C(\dR)=\C\mathbf{\dot{+}}C_0(\R)$. 
It follows from the mean value theorem
that
\begin{equation}\label{eq:embedding-of-classes-of-conitnuous-functions}
\C\mathbf{\dot{+}}C_c^\infty(\R)\subset C(\dR)\cap V(\R).
\end{equation}

Suppose $X(\R)$ is a separable Banach function space such that the 
Hardy-Little\-wood maximal operator $M$ is bounded on $X(\R)$
and on its associate space $X'(\R)$. Along with the algebra
$C_X(\dR)$ of continuous Fourier multipliers defined by
\eqref{eq:algebra-CX}, consider the following algebra of continuous 
Fourier multipliers:
\begin{equation}\label{eq:algebra-CX0}
C_{X}^0(\dR):=\operatorname{clos}_{\cM_{X(\R)}}
\big(\C\mathbf{\dot{+}}C_c^\infty(\R)\big).
\end{equation}

It follows from embeddings 
\eqref{eq:embedding-of-classes-of-conitnuous-functions} and definitions 
\eqref{eq:algebra-CX} and \eqref{eq:algebra-CX0} that
\begin{equation}\label{eq:embeddings-of-classes-of-continuous-multipliers}
C_{X}^0(\dR)\subset C_{X}(\dR).
\end{equation}

For large classes of Banach function spaces, including separable
re\-ar\-range\-ment-invariant Banach function with 
{nontrivial} Boyd indices,
weighted Lebesgue spaces with Muckenhoupt weights, reflexive variable 
Lebesgue spaces $L^{p(\cdot)}(\R)$ such that the Hardy-Littlewood
maximal operator $M$ is bounded on $L^{p(\cdot)}(\R)$, the above
embedding becomes equality (see \cite[Theorem~3.3]{FK20} and
\cite[Theorem~1.1]{K20}). Proofs of \cite[Theorem~3.3]{FK20}
and \cite[Theorem~1.1]{K20} are based on an 
{interpolation} argument. 
Unfortunately, interpolation tools are not available in the general setting
of Banach function spaces. So, we arrive at the following.
\begin{question}\label{question-1}
Is it true that $C_{X}^0(\dR)= C_{X}(\dR)$ for an arbitrary
separable Banach function space $X(\R)$ such that the Hardy-Littlewood
maximal operator is bounded on $X(\R)$ and on its associate space
$X'(\R)$?
\end{question}
\subsection{The ideal of compact operators is contained in the 
algebra of convolution type operators with continuous data}
Since we do not know the answer on Question~\ref{question-1}, 
along with the Banach algebra $\cA_{X(\R)}$, we will also consider
the smallest Banach subalgebra
\[
\cA_{X(\R)}^0
:=
\operatorname{alg}\{aI,W^0(b)\ :\ a\in C(\dR),\ b\in C_X^0(\dR)\}
\]
of the algebra $\cB(X(\R))$ that {contains} 
all operators of multiplication $aI$ 
by functions $a\in C(\dR)$ and all Fourier convolution operators $W^0(b)$ with 
symbols $b\in C_X^0(\dR)$.

If the answer to Question~\ref{question-1} is negative, then 
the following result provides a refinement of Theorem~\ref{th:FKK}.
\begin{theorem}\label{th:FKK-refined}
Let $X(\R)$ be a reflexive Banach function space such that the Hardy-Littlewood 
maximal operator $M$ is bounded on $X(\R)$ and on its associate 
space $X'(\R)$.  Then the ideal of compact operators $\cK(X(\R))$ is contained 
in the Banach algebra $\cA_{X(\R)}^0$.
\end{theorem}
The proof of Theorem~\ref{th:FKK-refined} repeats word-by-word the proof
of Theorem~\ref{th:FKK} with \cite[Lemma~4.2]{FKK19} replaced by the 
following.
\begin{lemma}\label{le:one-dimensional-operator}
Let $X(\R)$ be a separable Banach function space such that the Hardy-Littlewood 
maximal operator $M$ is bounded on $X(\R)$ and on its associate 
space $X'(\R)$. Suppose $a,b\in C_c(\R)$ and a
one-dimensional operator $T_1$ is defined on the space $X(\R)$ by
\begin{equation}\label{eq:one-dimensional-operator-1}
(T_1f)(x)=a(x)\int_\R b(y)f(y)\,dy.
\end{equation}
Then there exists a function $c\in C_{X}^0(\dR)$ 
such that $T_1=aW^0(c)bI$.
\end{lemma}
\begin{proof}
The idea of the proof is borrowed from \cite[Lemma~6.1]{KILH13a}
(see also \cite[Proposition~5.8.1]{RSS11}). Since $a,b\in C_c(\R)$, there
exists a number $M>0$ such that the set $\{x-y:x\in\operatorname{supp}a,
y\in\operatorname{supp}b\}$ is contained in the segment {$[-M,M]$.} 
By the smooth version of Urysohn's lemma (see, e.g.,
\cite[Proposition~6.5]{F09}), there exists $k\in C_c^\infty(\R)$
such that $0\le k(x)\le 1$ for $x\in\R$, $k(x)=1$ for $x\in[-M,M]$ and $k(x)=0$
for $x\in\R\setminus(-2M,2M)$. Then \eqref{eq:one-dimensional-operator-1}
can be rewritten in the form
\[
(T_1f)(x)=a(x)\int_\R k(x-y)b(y)f(y)\,dy=\big(aW^0(\widehat{k})bf\big)(x),
\quad
x\in\R.
\]
{
By \cite[Example~2.2.2 and Proposition~2.2.11]{G14},
$C_c^\infty(\R)\subset\mathcal{S}(\R)$ and $\widehat{k}\in\mathcal{S}(\R)$.
Since $\mathcal{S}(\R)\subset C_0(\R)\cap V(\R)$,
it follows from Lemma~\ref{le:approximating-vanishing-multipliers}(a)
that
\[
\mathcal{S}(\R)\subset
\operatorname{clos}_{\cM_{X(\R)}}\big(C_c^\infty(\R)\big)\subset C_X^0(\dR).
\]
Hence $c:=\widehat{k}\in C_X^0(\dR)$.
}
\end{proof}
\subsection{Slowly oscillating Fourier multipliers}
\label{sec:SO-multipliers}
For a set $E\subset\dR$ and a function
$f:\dR\to\C$ in $L^\infty(\R)$, let the oscillation of
$f$ over $E$ be defined by
\[
\operatorname{osc}(f,E)
:=
\operatornamewithlimits{ess\,sup}_{s,t\in E}|f(s)-f(t)|.
\]
Following \cite[Section~4]{BFK06} and
\cite[Section~2.1]{KILH12}, \cite[Section~2.1]{KILH13a},
we say that a function $f\in L^\infty(\R)$ is slowly
oscillating at a point $\lambda\in\dR$ if for every $r\in(0,1)$ or,
equivalently, for some $r\in(0,1)$, one has
\[
\begin{array}{lll}
\lim\limits_{x\to 0+}
\operatorname{osc}\big(f,\lambda+([-x,-rx]\cup[rx,x])\big)=0
&\mbox{if}& \lambda\in\R,
\\
\lim\limits_{x\to +\infty}
\operatorname{osc}\big(f,[-x,-rx]\cup[rx,x]\big)=0
&\mbox{if}& \lambda=\infty.
\end{array}
\]
For every $\lambda\in\dR$, let $SO_\lambda$ denote the $C^*$-subalgebra of
$L^\infty(\R)$ defined by
\[
SO_\lambda:=\left\{f\in C_b(\dR\setminus\{\lambda\})\ :\ f
\mbox{ slowly oscillates at }\lambda\right\},
\]
where
$C_b(\dR\setminus\{\lambda\}):=C(\dR\setminus\{\lambda\})\cap L^\infty(\R)$.

Let $SO^\diamond$ be the smallest $C^*$-subalgebra of $L^\infty(\R)$ that
contains all the $C^*$-algebras $SO_\lambda$ with $\lambda\in\dR$.
The functions in $SO^\diamond$ are called slowly oscillating functions.

For a point $\lambda\in\dR$, let $C^3(\R\setminus\{\lambda\})$ be the set of
all three times continuously differentiable functions
$a:\R\setminus\{\lambda\}\to\C$.
Following \cite[Section~2.4]{KILH12} and \cite[Section~2.3]{KILH13a}, consider
the commutative Banach algebras
\[
SO_\lambda^3:=\left\{
a\in SO_\lambda\cap C^3(\R\setminus\{\lambda\})\ :\
\lim_{x\to\lambda}(D_\lambda^k a)(x)=0,
\ k=1,2,3
\right\}
\]
equipped with the norm
\[
\|a\|_{SO_\lambda^3}:=
\sum_{{k=0}}^3\frac{1}{{k!}}
\left\|D_\lambda^ka\right\|_{L^\infty(\R)},
\]
where $(D_\lambda a)(x)=(x-\lambda) a'(x)$ for $\lambda\in\R$ and
$(D_\lambda a)(x)=xa'(x)$ for $\lambda=\infty$.

The following result leads us to the definition of slowly oscillating
Fourier multipliers.
\begin{theorem}[{\cite[Theorem~2.5]{K15c}}]
\label{th:boundedness-convolution-SO}
Let $X(\R)$ be a separable Banach function space such that the
Hardy-Littlewood maximal operator $M$ is bounded on $X(\R)$ and on its
associate space $X'(\R)$. If $\lambda\in\dR$ and $a\in SO_\lambda^3$, then
the convolution operator $W^0(a)$ is bounded on the space $X(\R)$ and
\[
\|W^0(a)\|_{\cB(X(\R))}
\le
d_{X}\|a\|_{SO_\lambda^3},
\]
where $d_{X}$ is a positive constant depending only on $X(\R)$.
\end{theorem}
Let $SO_{\lambda,X(\R)}$ denote the closure of $SO_\lambda^3$ in the norm of
$\cM_{X(\R)}$. Further, let $SO_{X(\R)}^\diamond$ be the smallest Banach
subalgebra of $\cM_{X(\R)}$ that contains all the Banach algebras
$SO_{\lambda,X(\R)}$ for $\lambda\in\dR$. The functions in
$SO_{X(\R)}^\diamond$ will be called slowly oscillating Fourier multipliers.
\subsection{The ideal of compact operators is contained in the 
algebra of convolution type operators with slowly oscillating data}
Consider the smallest Banach subalgebra
\[
\mathcal{D}_{X(\R)}
:=
\operatorname{alg}\{aI,W^0(b)\ :\ a\in SO^\diamond,\ b\in SO^\diamond_{X(\R)}\}
\]
of the algebra $\cB(X(\R))$ that {contains} 
all operators of multiplication $aI$ 
by slowly oscillating functions $a\in SO^\diamond$ and all Fourier convolution
operators $W^0(b)$ with slowly oscillating symbols $b\in SO_{X(\R)}^\diamond$.

Now we are in a position to formulate the main result of this section.
\begin{theorem}\label{th:compacts-in-D}
Let $X(\R)$ be a reflexive Banach function space such that the Hardy-Littlewood 
maximal operator $M$ is bounded on $X(\R)$ and on its associate space 
$X'(\R)$. Then the ideal of compact operators $\cK(X(\R))$ is contained in the 
Banach algebra $\mathcal{D}_{X(\R)}$.
\end{theorem}
{
This result follows from Theorem~\ref{th:FKK-refined}.
}

Under the assumptions of Theorem~\ref{th:compacts-in-D}, we can define
the quotient algebra
\[
\mathcal{D}_{X(\R)}^\pi:=\mathcal{D}_{X(\R)}/\mathcal{K}(X(\R)).
\]
We conclude this section with the following.
\begin{question}
Let $X(\R)$ be a reflexive Banach function space such that the Hardy-Littlewood 
maximal operator $M$ is bounded on $X(\R)$ and on its associate space 
$X'(\R)$. Is it true that the quotient algebra $\mathcal{D}_{X(\R)}^\pi$
is commutative?
\end{question}
We know that the answer is positive for some particular cases of Banach 
function spaces.
For Lebesgue spaces $L^p(\R,w)$, $1<p<\infty$, with Muckenhoupt weights $w$,
the positive answer to the above question follows from 
\cite[Theorem~4.6]{KILH13a}, whose proof relies on a version of the 
Krasnosel'skii interpolation theorem for compact operators
(see, e.g., \cite[Corollary~5.3]{K12}).
The answer is also positive for reflexive variable Lebesgue spaces
$L^{p(\cdot)}(\R)$ such that the Hardy-Littlewood maximal operator $M$
is bounded on $L^{p(\cdot)}(\R)$. It is based on a similar interpolation
argument (see \cite[Lemma 6.4]{K15a}). However, as far as we know, for 
arbitrary Banach function spaces, interpolation tools are not available.
\subsection*{Acknowledgment}
This work was supported by national funds through the FCT – Funda\c{c}\~ao 
para a Ci\^encia e a Tecnologia, I.P. (Portuguese Foundation for Science and 
Technology) within the scope of the project UIDB/00297/2020 
(Centro de Matem\'atica e Aplicações).

We are grateful to Helena Mascarenhas, who asked the first author about
a possibility of refinement of Theorem~\ref{th:FKK} contained in 
Theorem~\ref{th:FKK-refined}.
{
We thank the anonymous referee for useful remarks.
}

\end{document}